\documentclass[11pt,leqno]{article}
\usepackage[margin=1in]{geometry} 
\usepackage{amssymb,amsfonts,amsmath,bbm,mathrsfs,stmaryrd,mathtools}
\usepackage{xcolor}
\usepackage{graphicx}
\usepackage{url}

\usepackage{extarrows}

\usepackage[shortlabels]{enumitem}
\usepackage{tensor}

\usepackage{xr}
\usepackage[T1]{fontenc}

\usepackage[colorlinks,
linkcolor=black!75!red,
citecolor=blue,
pdftitle={},
pdfproducer={pdfLaTeX},
pdfpagemode=None,
bookmarksopen=true,
bookmarksnumbered=true,
backref=page]{hyperref}

\usepackage{tikz}
\usetikzlibrary{arrows,calc,decorations.pathreplacing,decorations.markings,decorations.shapes,intersections,shapes.geometric,through,fit,shapes.symbols,positioning,decorations.pathmorphing}

\makeatletter
\DeclareRobustCommand\widecheck[1]{{\mathpalette\@widecheck{#1}}}
\def\@widecheck#1#2{%
    \setbox\z@\hbox{\m@th$#1#2$}%
    \setbox\tw@\hbox{\m@th$#1%
       \widehat{%
          \vrule\@width\z@\@height\ht\z@
          \vrule\@height\z@\@width\wd\z@}$}%
    \dp\tw@-\ht\z@
    \@tempdima\ht\z@ \advance\@tempdima2\ht\tw@ \divide\@tempdima\thr@@
    \setbox\tw@\hbox{%
       \raise\@tempdima\hbox{\scalebox{1}[-1]{\lower\@tempdima\box
\tw@}}}%
    {\ooalign{\box\tw@ \cr \box\z@}}}
\makeatother

\usepackage{braket}

\usepackage[amsmath,thmmarks,hyperref]{ntheorem}
\usepackage{cleveref}

\newcommand\nthalias[1]{\AddToHook{env/#1/begin}{\crefalias{lemma}{#1}}}

\nthalias{definition}
\nthalias{example}
\nthalias{examples}
\nthalias{remark}
\nthalias{remarks}
\nthalias{convention}
\nthalias{notation}
\nthalias{construction}
\nthalias{sketch}
\nthalias{theoremN}
\nthalias{propositionN}
\nthalias{corollaryN}
\nthalias{lemma}
\nthalias{proposition}
\nthalias{corollary}
\nthalias{theorem}
\nthalias{conjecture}
\nthalias{question}
\nthalias{assumption}
\nthalias{recollection}

\creflabelformat{enumi}{#2#1#3}

\crefname{section}{Section}{Sections}
\crefformat{section}{#2Section~#1#3} 
\Crefformat{section}{#2Section~#1#3} 

\crefname{subsection}{\S}{\S\S}
\AtBeginDocument{%
  \crefformat{subsection}{#2\S#1#3}%
  \Crefformat{subsection}{#2\S#1#3}%
}

\crefname{subsubsection}{\S}{\S\S}
\AtBeginDocument{%
  \crefformat{subsubsection}{#2\S#1#3}%
  \Crefformat{subsubsection}{#2\S#1#3}%
}

\theoremstyle{plain}

\newtheorem{lemma}{Lemma}[section]
\newtheorem{proposition}[lemma]{Proposition}
\newtheorem{corollary}[lemma]{Corollary}
\newtheorem{theorem}[lemma]{Theorem}

\theoremstyle{plain}
\theoremnumbering{Alph}

\theoremstyle{plain}
\theorembodyfont{\upshape}
\theoremsymbol{\ensuremath{\blacklozenge}}

\newtheorem{example}[lemma]{Example}

\newtheorem{remarks}[lemma]{Remarks}

\crefname{definition}{definition}{definitions}
\crefformat{definition}{#2definition~#1#3} 
\Crefformat{definition}{#2Definition~#1#3} 

\crefname{ex}{example}{examples}
\crefformat{example}{#2example~#1#3} 
\Crefformat{example}{#2Example~#1#3} 

\crefname{exs}{example}{examples}
\crefformat{examples}{#2example~#1#3} 
\Crefformat{examples}{#2Example~#1#3} 

\crefname{remark}{remark}{remarks}
\crefformat{remark}{#2remark~#1#3} 
\Crefformat{remark}{#2Remark~#1#3} 

\crefname{remarks}{remark}{remarks}
\crefformat{remarks}{#2remark~#1#3} 
\Crefformat{remarks}{#2Remark~#1#3} 

\crefname{convention}{convention}{conventions}
\crefformat{convention}{#2convention~#1#3} 
\Crefformat{convention}{#2Convention~#1#3} 

\crefname{notation}{notation}{notations}
\crefformat{notation}{#2notation~#1#3} 
\Crefformat{notation}{#2Notation~#1#3} 

\crefname{table}{table}{tables}
\crefformat{table}{#2table~#1#3} 
\Crefformat{table}{#2Table~#1#3}

\crefname{lemma}{lemma}{lemmas}
\crefformat{lemma}{#2lemma~#1#3} 
\Crefformat{lemma}{#2Lemma~#1#3} 

\crefname{proposition}{proposition}{propositions}
\crefformat{proposition}{#2proposition~#1#3} 
\Crefformat{proposition}{#2Proposition~#1#3} 

\crefname{propositionN}{proposition}{propositions}
\crefformat{propositionN}{#2proposition~#1#3} 
\Crefformat{propositionN}{#2Proposition~#1#3} 

\crefname{corollary}{corollary}{corollaries}
\crefformat{corollary}{#2corollary~#1#3} 
\Crefformat{corollary}{#2Corollary~#1#3} 

\crefname{corollaryN}{corollary}{corollaries}
\crefformat{corollaryN}{#2corollary~#1#3} 
\Crefformat{corollaryN}{#2Corollary~#1#3} 

\crefname{theorem}{theorem}{theorems}
\crefformat{theorem}{#2theorem~#1#3} 
\Crefformat{theorem}{#2Theorem~#1#3} 

\crefname{theoremN}{theorem}{theorems}
\crefformat{theoremN}{#2theorem~#1#3} 
\Crefformat{theoremN}{#2Theorem~#1#3} 

\crefname{enumi}{}{}
\crefformat{enumi}{#2#1#3}
\Crefformat{enumi}{#2#1#3}

\crefname{assumption}{assumption}{Assumptions}
\crefformat{assumption}{#2assumption~#1#3} 
\Crefformat{assumption}{#2Assumption~#1#3} 

\crefname{construction}{construction}{Constructions}
\crefformat{construction}{#2construction~#1#3} 
\Crefformat{construction}{#2Construction~#1#3} 

\crefname{sketch}{sketch}{Sketches}
\crefformat{sketch}{#2sketch~#1#3} 
\Crefformat{sketch}{#2Sketch~#1#3} 

\crefname{recollection}{recollection}{Recollectiones}
\crefformat{recollection}{#2recollection~#1#3} 
\Crefformat{recollection}{#2Recollection~#1#3} 

\crefname{question}{question}{Questions}
\crefformat{question}{#2question~#1#3} 
\Crefformat{question}{#2Question~#1#3} 

\crefname{equation}{}{}
\crefformat{equation}{(#2#1#3)} 
\Crefformat{equation}{(#2#1#3)}

\numberwithin{equation}{section}

\theoremstyle{nonumberplain}
\theoremsymbol{\ensuremath{\blacksquare}}

\newtheorem{proof}{Proof}
\newcommand\pf[1]{\newtheorem{#1}{Proof of \Cref{#1}}}

\newcommand\bG{{\mathbb G}}

\newcommand\bK{{\mathbb K}}

\newcommand\bQ{{\mathbb Q}}
\newcommand\bR{{\mathbb R}}

\newcommand\bZ{{\mathbb Z}}

\newcommand\cC{{\mathcal C}}

\newcommand\cN{{\mathcal N}}

\newcommand\cP{{\mathcal P}}

\newcommand\cU{{\mathcal U}}

\newcommand\fc{{\mathfrak c}}

\newcommand{\wtimes}{\mathbin{\widetilde{\times}}}

\DeclareMathOperator{\id}{id}

\newcommand{\cat}[1]{\textsc{#1}}

\newcommand{\qedhere}{\mbox{}\hfill\ensuremath{\blacksquare}}

\newcommand{\xrightarrowdbl}[2][]{%
  \xrightarrow[#1]{#2}\mathrel{\mkern-14mu}\rightarrow
}

\title{Discontinuous actions on cones, joins, and $n$-universal bundles}
\author{Alexandru Chirvasitu}

\begin{document}

\date{}

\newcommand{\Addresses}{{
  \bigskip
  \footnotesize

  \textsc{Department of Mathematics, University at Buffalo}
  \par\nopagebreak
  \textsc{Buffalo, NY 14260-2900, USA}  
  \par\nopagebreak
  \textit{E-mail address}: \texttt{achirvas@buffalo.edu}

}}

\maketitle

\begin{abstract}
  We prove that locally countably-compact Hausdorff topological groups $\mathbb{G}$ act continuously on their iterated joins $E_n\mathbb{G}:=\mathbb{G}^{*(n+1)}$ (the total spaces of the Milnor-model $n$-universal $\mathbb{G}$-bundles) as well as the colimit-topologized unions $E\mathbb{G}=\varinjlim_n E_n\mathbb{G}$, and the converse holds under the assumption that $\mathbb{G}$ is first-countable. In the latter case other mutually equivalent conditions provide characterizations of local countable compactness: the fact that $\mathbb{G}$ acts continuously on its first self-join $E_1\mathbb{G}$, or on its cone $\mathcal{C}\mathbb{G}$, or the coincidence of the product and quotient topologies on $\mathbb{G}\times \mathcal{C}X$ for all spaces $X$ or, equivalently, for the discrete countably-infinite $X:=\aleph_0$. These can all be regarded as weakened versions of $\mathbb{G}$'s exponentiability, all to the effect that $\mathbb{G}\times -$ preserves certain colimit shapes in the category of topological spaces; the results thus extend the equivalence (under the separation assumption) between local compactness and exponentiability. 
\end{abstract}


\vspace{.5cm}

\noindent{MSC 2020: 22F05; 54B15; 54D20; 18A30; 06F20; 54D15; 06F30; 03C20
  
}


\section*{Introduction}

For a topological group $\bG$, there will be frequent references to Milnor's \emph{universal principal $\bG$-bundle} \cite[\S 3]{miln_univ-2}
\begin{equation*}
  \left(E\bG:=\bigcup_n E_n \bG\right)
  \xrightarrowdbl{\quad}
  \left(B\bG:=\bigcup_n B_n \bG\right)
  ,\quad
  E_n\bG:=\bG^{*(n+1)}
  ,\quad
  B_n\bG:=E_n\bG/\bG
\end{equation*}
where
\begingroup
\allowdisplaybreaks
\begin{align*}
  X*Y
  &:=
    X\times Y\times \left(I:=[0,1]\right)
    \bigg/
    \begin{aligned}
      (x,y,0)&\sim (x,y',0)\\
      (x,y,1)&\sim (x',y,1)
    \end{aligned}\\
  \cC X&:=X\times I\big/X\times \{0\}
\end{align*}
\endgroup are the \emph{join} of $X$ and $Y$ and the \emph{cone} on $X$ respectively (\cite[pp.9-10]{hatch_at}, \cite[\S 2]{miln_univ-2}) and the (free) $\bG$-actions on $E_n\bG$, $n\le \infty$ are the obvious translation ones. 

The importance of $E\bG$ (and analogous constructions such as those employed in \cite[\S 16.5]{may_at_1999} or \cite[\S 7]{may_cls}) lies in its \emph{universality}: locally trivial \emph{numerable} \cite[\S 14.3]{td_alg-top} principal $\bG$-bundles over $X$ are classified \cite[Theorem 14.4.1]{td_alg-top} as pullbacks along maps $X\to B\bG$ uniquely defined up to homotopy. Cones, joins and $E\bG$ each carry at least two topologies of interest in that context. Focusing on the more elaborate construct $E\bG$ (with the notation extending to cones and joins as well), there is
\begin{enumerate}[(a),wide]
\item a \emph{stronger} topology $\tau_{\varinjlim}$, in the usual sense \cite[Definition 3.1]{wil_top} of having more open sets, defined by equipping every quotient in sight (so all joins) with the respective \emph{quotient topology} \cite[Definition 9.1]{wil_top} and then regarding $E\bG=\bigcup_n E_n\bG$ as a \emph{colimit} in the category of topological spaces;

\item a \emph{weaker} topology $\tau_w$ (meaning\footnote{It is somewhat unfortunate that the terms `weak' and `strong', in the present context of comparing topologies, appear to have had their meanings precisely interchanged: \cite[\S\S 2 and 5]{miln_univ-2}, for instance, employ them in exactly opposite fashion.} \cite[Definition.3.1]{wil_top} \emph{fewer} open sets) obtained \cite[\S 14.4, Problem 10]{td_alg-top} by embedding
  \begin{equation*}
    E\bG
    \cong
    \left\{(t_n g_n)_n\ :\ t_ng_n\in \cC\bG\ \wedge \sum_n t_n=1\right\}
    \subseteq
    \left(\cC\bG\right)^{\bZ_{\ge 0}},
  \end{equation*}
  and equipping each cone $\cC \bG$ with its \emph{coordinate topology}: weakest with all
  \begin{equation*}
    \cC X\ni tx
    \xmapsto{\quad\tau\quad}
    t\in
    I
    ,\quad
    \tau^{-1}\left((0,1]\right)
    \ni tx
    \xmapsto{\quad}
    x\in X
  \end{equation*}
  continuous.
\end{enumerate}

Equipped with the weaker topology $\tau_w$, $E\bG$ is indeed a contractible $\bG$-space (the latter phrase meaning, here, ``topological space equipped with a continuous $\bG$-action''), and textbook accounts tend to proceed on these lines: \cite[\S 14.4.3]{td_alg-top} or \cite[\S 4.11]{hus_fib}, say. In settings where $\tau_{\varinjlim}$ is preferred, the heart of the matter seems to be the continuity of the resulting $\bG$-action. While counterexamples are easily produced to illustrate its failure ($(\bQ,+)$, for instance, already acts discontinuously on its \emph{first} self-join $E_1\bQ=\bQ*\bQ$ \cite[Proposition 2.2]{2511.13511v1}), various devices can mitigate such pathologies.
\begin{itemize}[wide]
\item In first instance, if the topological group $\bG$ is well-behaved enough, the continuity of
  \begin{equation}\label{eq:act.on.colim}
    \bG \times \left(E\bG,\varinjlim\right) 
    \xrightarrow{\quad}
    \left(E\bG,\varinjlim\right)
  \end{equation}
  is automatic. Specifically, it suffices that $\bG$ be locally compact: according to \cite[Proposition 7.1.5]{brcx_hndbk-2} it is in that case \emph{exponentiable} in the sense \cite[Definition 7.1.3]{brcx_hndbk-2} that $-\times \bG$ is a left adjoint on the category of topological spaces, so is \emph{cocontinuous} \cite[dualized Proposition 3.2.2]{brcx_hndbk-1} (preserves colimits). The domain of \Cref{eq:act.on.colim} thus itself carries a colimit topology, hence the continuity of the action by colimit functoriality. 

  This gadgetry is operative in the setting of \cite[\S 2]{atiyahsegal}, say, where the colimit topology is employed and \cite[footnote 1]{atiyahsegal} points out why all is still well because groups are assumed, there, compact Lie. 
  
\item Alternatively, some sources take the somewhat more sophisticated route of substituting for the usual Cartesian-product-equipped category $\left(\cat{Top},\times\right)$ of topological spaces that of (Hausdorff) \emph{compactly-generated (or $k$-)spaces} (\cite[Definition 43.8]{wil_top}, \cite[Definition 7.2.5]{brcx_hndbk-2}), equipped with \emph{its} categorical product $\times_k$.

  The \emph{Cartesian closure} \cite[Corollary 7.2.6]{brcx_hndbk-2} of the latter category $\left(\cat{Top}_k,\times_k\right)$ then ensures that all endofunctors $-\times_k X$ are left adjoints, so the previous item's argument applies universally and \Cref{eq:act.on.colim} is continuous if $\bG$ is a group object internal to $\left(\cat{Top}_k,\times_k\right)$ and $\times$ is reinterpreted as $\times_k$. This is the machinery at work in \cite{MR353298} (per \cite[\S 0, very last sentence]{MR353298}) or \cite[\S 5]{may_at_1999}, for instance. In the latter case this is on first sight somewhat obscured by the presentation, but the geometric-realization constructions employed in \cite[\S 5]{may_at_1999} applies the material developed in \cite[\S 4]{may_at_1999}, which in turn takes for its base category a slightly broader analogue of $\left(\cat{Top}_k,\times_k\right)$ (in the sense that the Hausdorff condition is somewhat relaxed; see the conventions spelled out in \cite[\S 5.2]{may_at_1999}). 
\end{itemize}

\cite[Definition 7.2.7]{hjjm_bdle} seems to be an exception to this dichotomy in approaches to the issue of continuity in \Cref{eq:act.on.colim}: while the quotient topology is adopted, neither constraints on $\bG$ nor ``non-standard'' ambient categories of topological spaces appear to be in place in that discussion. 

These subtleties in how careful one must be, and under what conditions, in order to ensure the continuity of \Cref{eq:act.on.colim} (or rather, below, of its truncated versions $\bG \times E_n\bG\xrightarrow{\triangleright_n} E_n\bG$) are what motivate and form the focus of the paper. Specifically, the main result, giving in particular a full characterization of those first-countable groups for which all $\triangleright_n$ are continuous, is as follows.

\begin{theorem}\label{th:lc.a0.cpct}
  Consider the following conditions on a Hausdorff topological group.
  \begin{enumerate}[(a),wide]

  \item\label{item:th:lc.a0.cpct:lcc} $\bG$ is locally countably-compact. 

  \item\label{item:th:lc.a0.cpct:eg} The $\bG$-action on the colimit-topologized full Milnor total space $E\bG=\varinjlim_n E_n\bG$ is continuous.
    
  \item\label{item:th:lc.a0.cpct:all.l} For any (possibly non-$T_2$) topological space $X$ the left-hand translation $\bG$-action on the quotient-topologized
    \begin{equation*}
      \left(
        \text{quotient-topologized }
        \bG\times X\times I
        \xrightarrowdbl{\ }
        \bG\times \cC X
      \right)
      =:
      \bG\wtimes \cC X
    \end{equation*}
    is continuous.
    
  \item\label{item:th:lc.a0.cpct:all.same.top} For any (possibly non-$T_2$) topological space $X$ the identity
    \begin{equation*}
      \bG\wtimes \cC X
      \xrightarrow{\ \id\ }
      \left(\bG\times \cC X,\ \text{product topology}\right)
    \end{equation*}
    is a homeomorphism.  

  \item\label{item:th:lc.a0.cpct:en} For all (some) $n\in \bZ_{\ge 1}$ the $\bG$-action on the colimit-topologized truncated Milnor total space $E_n\bG:=\bG^{*(n+1)}$ is continuous.
    
  \item\label{item:th:lc.a0.cpct:e1} The $\bG$-action on the truncated Milnor total space $E_1\bG:=\bG*\bG$ with its colimit topology is continuous.

  \item\label{item:th:lc.a0.cpct:gcg.diag} The diagonal $\bG$-action on the quotient-topologized Cartesian product $\bG\wtimes \cC\bG$ is continuous. 

  \item\label{item:th:lc.a0.cpct:gcg.l} The left-hand translation $\bG$-action on $\bG\wtimes \cC\bG$ is continuous.

  \item\label{item:th:lc.a0.cpct:same.top} $\bG\wtimes \cC\bG \xrightarrow{\id} \bG\times \cC\bG$ is a homeomorphism.
    
  \item\label{item:th:lc.a0.cpct:gca0.l} The left-hand translation $\bG$-action on $\bG\wtimes \cC\aleph_0$ is continuous, $\aleph_0$ denoting a countably-infinite discrete space. 

  \item\label{item:th:lc.a0.cpct:a0.same.top} $\bG\wtimes \cC\aleph_0 \xrightarrow{\id} \bG\times \cC\aleph_0$ is a homeomorphism.
    
  \item\label{item:th:lc.a0.cpct:cn} The $\bG$-action on the cone $\cC \bG$ with its quotient topology is continuous.  
  \end{enumerate}
  The implications
  \begin{equation*}    
    \begin{tikzpicture}[>=stealth,auto,baseline=(current  bounding  box.center)]
      \draw[anchor=base] (-3.6,0) node (ll) {\Cref{item:th:lc.a0.cpct:lcc}};
      \draw[anchor=base] (-1.8,-.5) node (eg) {\Cref{item:th:lc.a0.cpct:eg}};
      \draw[anchor=base] (-2.4,.5) node (l) {\Cref{item:th:lc.a0.cpct:all.l}};
      \draw[anchor=base] (-1.2,.5) node (1) {\Cref{item:th:lc.a0.cpct:all.same.top}};
      \draw[anchor=base] (0,0) node (15) {\Cref{item:th:lc.a0.cpct:en}};
      \draw[anchor=base] (1.2,0) node (2) {\Cref{item:th:lc.a0.cpct:e1}};
      \draw[anchor=base] (2.4,0) node (3) {\Cref{item:th:lc.a0.cpct:gcg.diag}};
      \draw[anchor=base] (3.6,0) node (4) {\Cref{item:th:lc.a0.cpct:gcg.l}};
      \draw[anchor=base] (4.8,0) node (5) {\Cref{item:th:lc.a0.cpct:same.top}};
      \draw[anchor=base] (6,.5) node (6) {\Cref{item:th:lc.a0.cpct:gca0.l}};
      \draw[anchor=base] (7.2,.5) node (7) {\Cref{item:th:lc.a0.cpct:a0.same.top}};
      \draw[anchor=base] (6,-.5) node (8) {\Cref{item:th:lc.a0.cpct:cn}};

      \draw[-implies,double equal sign distance] (ll) to[bend right=6] node[pos=.5,auto] {$\scriptstyle $} (eg);
      \draw[-implies,double equal sign distance] (eg) to[bend right=6] node[pos=.5,auto] {$\scriptstyle $} (15);      
      \draw[-implies,double equal sign distance] (ll) to[bend left=16] node[pos=.5,auto] {$\scriptstyle $} (l);
      \draw[implies-implies,double equal sign distance] (l) to[bend left=16] node[pos=.5,auto] {$\scriptstyle $} (1);
      \draw[-implies,double equal sign distance] (1) to[bend left=16] node[pos=.5,auto] {$\scriptstyle $} (15);
      \draw[-implies,double equal sign distance] (15) to[bend left=0] node[pos=.5,auto] {$\scriptstyle $} (2);
      \draw[implies-implies,double equal sign distance] (2) to[bend left=0] node[pos=.5,auto] {$\scriptstyle $} (3);
      \draw[implies-implies,double equal sign distance] (3) to[bend left=0] node[pos=.5,auto] {$\scriptstyle $} (4);
      \draw[implies-implies,double equal sign distance] (4) to[bend left=0] node[pos=.5,auto] {$\scriptstyle $} (5);
      \draw[-implies,double equal sign distance] (5) to[bend left=6] node[pos=.5,auto] {$\scriptstyle $} (6);
      \draw[implies-implies,double equal sign distance] (6) to[bend left=6] node[pos=.5,auto] {$\scriptstyle $} (7);
      \draw[-implies,double equal sign distance] (5) to[bend right=6] node[pos=.5,auto] {$\scriptstyle $} (8);
    \end{tikzpicture}
  \end{equation*}
  hold, and for first-countable $\bG$ all conditions are mutually equivalent. 
\end{theorem}

\emph{First-countability} means, as usual \cite[Definition 10.3]{wil_top}, that points have countable neighborhood bases. Some of this appears in various guises in \cite[Proposition 2.2 and Theorem 2.8]{2511.13511v1} (and their proofs). The latter, for instance, proves (as part of a more elaborate statement) the implication \Cref{item:th:lc.a0.cpct:lcc} $\Rightarrow$ \Cref{item:th:lc.a0.cpct:all.same.top} of \Cref{th:lc.a0.cpct} for \emph{globally} countably-compact groups.

\subsection*{Acknowledgments}

Insightful input from M. Tobolski has contributed to improving various early drafts of the paper. This work is part of the project Graph Algebras partially supported by EU grant HORIZON-MSCA-SE-2021 Project 101086394. 


\section{Main result and partial exponentiability in various guises}\label{se:main}

Topological spaces default to being Hausdorff (or $T_2$, in familiar \cite[\S 13]{wil_top} separation-axiom-hierarchy terminology), with exceptions highlighted explicitly. For topological groups this in particular entails \cite[\S 33, Exercise 10]{mnk} \emph{complete regularity} (or the property of being \emph{Tychonoff}, or $T_{3\frac 12}$ \cite[Definition 14.8]{wil_top}), i.e. (Hausdorff+) continuous functions separate points and closed sets in the sense that
\begin{equation*}
  \forall\left(x\not\in \text{closed }A\subseteq X\right)
  \exists\left(X\xrightarrow[\text{continuous}]{\quad f\quad}\bR\right)
  \left(f(x)=1 \wedge f|_A\equiv 0\right).
\end{equation*}
Much as in \cite[post Theorem 1.3]{MR1326826}, for a property $\cP$ a space is \emph{locally $\cP$} if every point has a neighborhood satisfying $\cP$; this applies for instance to $\cP$ being
\begin{itemize}[wide]
\item compactness;

\item \emph{countable compactness} \cite[p.19]{ss_countertop} (every countable open cover having a finite subcover);

\item \emph{pseudocompactness} \cite[p.20]{ss_countertop} (real-valued continuous functions on the space are bounded);

\item or \emph{boundedness}\footnote{Not to be confused with other notions, non-specific to groups: post \cite[Corollary 6.9.5]{at_top-gp_2008}, for instance, a subset of $X$ is bounded if continuous real-valued functions on $X$ restrict to bounded functions thereon; the notion is certainly not equivalent to that in use here, as follows, say, from \cite[Proposition 6.9.26]{at_top-gp_2008} (which would not hold in the present context).} \cite[p.267]{MR1326826} for subsets of topological groups: admitting covers by finitely many translates of any identity neighborhood.
\end{itemize}

A few auxiliary observations will help streamline various portions of the proof of \Cref{th:lc.a0.cpct}. As a preamble to \Cref{pr:cpctk} below (appealed to in the proof of \Cref{th:lc.a0.cpct}'s \Cref{item:th:lc.a0.cpct:lcc} $\Rightarrow$ \Cref{item:th:lc.a0.cpct:all.same.top} implication), we collect some reminders and vocabulary.
\begin{itemize}[wide]
\item Recall \cite[Definition 1.14]{juh_card} that for topological subspaces $A\subseteq J$ a \emph{neighborhood base} (also \emph{local base}) of $A\subseteq J$ is a collection $\cU$ of neighborhoods $U\supseteq A$ in $J$ such that every neighborhood of $A$ contains some member of $\cU$ (i.e. $\cU$ is inclusion-\emph{dense} in the set of neighborhoods of $A\subseteq J$ in the usual order-theoretic sense \cite[Definition II.2.4]{kun_st}). 

\item Define \emph{characters}
  \begin{equation}\label{eq:chars}
    \chi(A,J)
    :=
    \min\left|\text{cardinality of a local base of $A\subseteq J$}\right|
    ,\quad
    \chi(J)
    :=
    \sup_{\text{points p}}\chi(p,J).
  \end{equation}

\item Given a condition $\cC$ on cardinal numbers, we refer to a space as \emph{compact$_{\cC}$} if every $\alpha$-member open cover has a finite subcover whenever $\alpha$ satisfies $\cC$. Taking $\cC$ to be empty recovers ordinary compactness, while compact$_{<\aleph_1}$ means countable compactness. The discussion centers mostly on compactness$<\kappa$.

\item For spaces $X$ topological embeddings $A\lhook\joinrel\xrightarrow{\iota} J$ we write
  \begin{equation*}
    \cC_{A\subseteq J} X
    =
    \cC_{\iota} X
    :=
    X\times J/X\times A,
  \end{equation*}
  equipping that space with its quotient topology unless specified otherwise. 
\end{itemize}

\begin{proposition}\label{pr:cpctk}
  Consider
  \begin{itemize}
  \item a $T_2$ compact$_{<\kappa}$ space $Z$ for an infinite cardinal $\kappa$;

  \item and a closed embedding $A\lhook\joinrel\xrightarrow{\iota} J$ with $\chi(A,J)<\kappa$. 
  \end{itemize}
  For arbitrary topological spaces $X$, the identity
  \begin{equation*}
    \left(
      \text{quotient-topologized }
      Z\times X\times J
      \xrightarrowdbl{\ \id_Z\times \pi\ }
      Z\times \cC_{\iota}X
    \right)
    =:
    Z\wtimes \cC_{\iota}X
    \xrightarrow{\ \id\ }
    \text{product }
    Z\times \cC_{\iota} X
  \end{equation*}
  is a homeomorphism.
\end{proposition}
\begin{proof}
  $A\subseteq J$ being closed, the topologies will in any case agree locally at points in the open (in either topology) complement of the image
  \begin{equation}\label{eq:copy.of.z}
    Z\cong \left(\id_Z\times \pi\right)\left(X\times A\right)
    \subseteq
    Z\times \cC_{\iota}X;
  \end{equation}
  it suffices to verify agreement of local bases around points belonging to \Cref{eq:copy.of.z}. For notational convenience, we move the discussion to the original space $Z\times X\times J$ and work with open subsets (or point neighborhoods) therein \emph{saturated} \cite[Definition 9.8]{wil_top} for the equivalence relation with the preimages of $\id_Z\times \pi$ as classes. 

  Consider, then, a saturated neighborhood
  \begin{equation*}
    U
    \supseteq
    z\times X\times A
    \subseteq
    Z\times X\times J
  \end{equation*}
  (suppressing braces from $\{z\}\times-$), which in particular contains a neighborhood of some $V_z\times X\times A$ for a neighborhood $V_z\ni z\in Z$ we may as well assume compact$_{<\kappa}$. For individual $x\in X$ the slice
  \begin{equation*}
    U|_x:=U\cap \left(Z\times x\times J\right)
  \end{equation*}
  contains $V_z\times A\cong V_z\times x\times A$; having fixed a local base $\left(W_{A,\lambda}\right)_{\lambda<\kappa'<\kappa}$ around $A\subseteq J$, $U|_x$ will contain a neighborhood of $V_z\times A\cong V_z\times x\times A$ of the form
  \begin{equation*}
    \bigcup_{\lambda<\kappa'}
    \left(
      V_{z,\lambda}\times W_{A,\lambda}
      \cong
      V_{z,\lambda}\times x\times W_{A,\lambda}
    \right)
    ,\quad
    \text{open }V_{z,\lambda}\ni z.
  \end{equation*}
  Compactness$_{<\kappa}$ ensures that finitely many $V_{z,\lambda}$ cover $V_z$, hence the existence of a neighborhood $X_x\ni x\in X$ with
  \begin{equation*}
    V_z\times X_x\times W_{A,\lambda=\lambda_x}
    \subseteq
    U.
  \end{equation*}
  Ranging over $x$,
  \begin{equation*}
    U
    \supseteq
    \bigcup_x V_z\times X_x\times W_{A,\lambda_x}
    =
    V_z\times \left(\bigcup_x X_x\times W_{A,\lambda_x}\right);
  \end{equation*}
  this confirms that $\left(\id_Z\times \pi\right)(U)$ is in fact a neighborhood of $z\in \text{\Cref{eq:copy.of.z}}$ in the Cartesian product $Z\times \cC_{\iota} X$ topologized as such. 
\end{proof}

\begin{remarks}\label{res:part.ladj.cpctk}
  \begin{enumerate}[(1),wide]
  \item As recalled post \Cref{eq:act.on.colim}, locally compact spaces (not necessarily $T_2$, if sufficient care is taken in defining the notion) are exponentiable and hence the corresponding endofunctors $-\times X$ are cocontinuous. \Cref{pr:cpctk} can be regarded as an analogue: it recovers a kind of partial cocontinuity given ``sufficient local compactness''.
    
  \item The term `$\kappa$-compact' might present itself as preferable to `compact$_{<\kappa}$', but it is already in use in the literature in several ways that conflict with the present intent: the notions employed in \cite[Definition 1.8]{juh_card} or \cite[2$^{nd}$ paragraph]{MR244947}, say (themselves mutually distinct) are such that increasing $\kappa$ produces a \emph{weaker} constraint; here, compactness$_{\kappa}$ is strength-wise \emph{non}-decreasing in $\kappa$. 
  \end{enumerate}
\end{remarks}

The following simple general remark underlies the equivalences \Cref{item:th:lc.a0.cpct:all.l} $\Leftrightarrow$ \Cref{item:th:lc.a0.cpct:all.same.top}, \Cref{item:th:lc.a0.cpct:gcg.l} $\Leftrightarrow$ \Cref{item:th:lc.a0.cpct:same.top} and \Cref{item:th:lc.a0.cpct:gca0.l} $\Leftrightarrow$ \Cref{item:th:lc.a0.cpct:a0.same.top} of \Cref{th:lc.a0.cpct}.

\begin{lemma}\label{le:act.on.prod.quot}
  Let $\bG$ be a topological group, $X$ a topological space, $R\subseteq X\times X$ an equivalence relation, and write
  \begin{itemize}[wide]
  \item $\bG\times X/R$ for the Cartesian product equipped with its usual product topology;
  \item and
    \begin{equation*}
      \bG\wtimes X/R
      :=
      \text{quotient-topologized }
      \left(
        \bG\times X
        \xrightarrowdbl{\quad}
        \bG\times X/R
      \right).
    \end{equation*}
  \end{itemize}
  The identity $\bG\wtimes X/R\to \bG\times X/R$ is a homeomorphism if and only if the left-translation action
  \begin{equation*}
    \bG
    \times
    \left(\bG\wtimes X/R\right)
    \xrightarrow{\quad}
    \bG\wtimes X/R
  \end{equation*}
  is continuous. 
\end{lemma}
\begin{proof}
  The forward implication $(\Rightarrow)$ is immediate, and the converse is effectively what \cite[Example 1.5.11]{dvr-bk} argues in the specific case $\bG:=(\bQ,+)$: the right-hand map in 
  \begin{equation*}
    \bG\times X/R
    \lhook\joinrel\xrightarrow{\quad\id\bG\times \left(\text{obvious embedding}\right)\quad}      
    \bG\times \left(\bG\wtimes X/R\right)
    \xrightarrow{\quad}
    \bG\wtimes X/R,
  \end{equation*}
  being assumed continuous, so is the composition. 
\end{proof}

To transition between the two types of actions mentioned in \Cref{th:lc.a0.cpct}\Cref{item:th:lc.a0.cpct:gcg.diag} and \Cref{item:th:lc.a0.cpct:gcg.l} we will need

\begin{lemma}\label{le:diag.l}
  Let $\bG$ be a topological group and $X$ a $\bG$-space (no separation assumptions), and consider the left-hand-translation and diagonal actions on $\bG\wtimes \cC X$.

  If one of those actions is continuous so is the other, and the resulting $\bG$-spaces are $\bG$-homeomorphic. 
\end{lemma}
\begin{proof}
  Simply observe that the self-homeomorphism
    \begin{equation}\label{eq:diag2l}
      \bG\times X\times I
      \ni
      (g,x,t)
      \xmapsto{\quad}
      (g,g^{-1}x,t)
      \in
      \bG\times X\times I
    \end{equation}
    intertwines the two $\bG$ actions in question and is compatible with the relation collapsing the Cartesian product onto $\bG\wtimes\cC X$.
\end{proof}

\pf{th:lc.a0.cpct}
\begin{th:lc.a0.cpct}
  The implications involving \Cref{item:th:lc.a0.cpct:en} are to be understood as making the strongest statements possible: the stronger version (all) follows from the properties claimed to be upstream, while the weakest version (some) implies those downstream. 
  
  \begin{enumerate}[label={},wide]
    
  \item\textbf{\Cref{item:th:lc.a0.cpct:all.l} $\Leftrightarrow$ \Cref{item:th:lc.a0.cpct:all.same.top}, \Cref{item:th:lc.a0.cpct:gcg.l} $\Leftrightarrow$ \Cref{item:th:lc.a0.cpct:same.top} and \Cref{item:th:lc.a0.cpct:gca0.l} $\Leftrightarrow$ \Cref{item:th:lc.a0.cpct:a0.same.top}:} Instances of a general observation relegated to \Cref{le:act.on.prod.quot}. 
    
  \item\textbf{\Cref{item:th:lc.a0.cpct:lcc} $\Rightarrow$ \Cref{item:th:lc.a0.cpct:all.same.top}:} A consequence of the broader phenomenon recorded in \Cref{pr:cpctk}. 
   
  \item\textbf{\Cref{item:th:lc.a0.cpct:all.same.top} $\Rightarrow$ \Cref{item:th:lc.a0.cpct:en}:} Cast $E_n\bG$ ($n\ge 1$) as a space
    \begin{equation*}
      E_n\bG
      =\left\{
        \sum_{i=0}^n t_i g_i
        \ :\
        g_i\in \bG
        ,\
        t_i\in [0,1]
        ,\
        \sum_i t_i=1
      \right\}
    \end{equation*}
    of convex combinations, covered for sufficiently small $\varepsilon>0$ by the interiors of its closed subspaces 
    \begin{equation}\label{eq:eieng}
      \tensor*[_{i\uparrow\varepsilon}]{E}{_n}\bG
      :=
      \left\{\sum t_i g_i\ :\ t_i\ge \varepsilon\right\}.
    \end{equation}
    It will thus suffice to prove the continuity of the $\bG$-action on a single $\tensor*[_{i\uparrow\varepsilon}]{E}{_n}\bG$, say for $i:=0$. There is a $\bG$-equivariant identification
    \begin{equation*}
      \tensor*[_{0\uparrow\varepsilon}]{E}{_n}\bG
      \ni
      \sum_{i=0}^n t_ig_i
      \xmapsto[\quad\cong\quad]{\quad}
      \left(
        g_0
        ,\ 
        t\sum_{i=1}^n s_i g_i
      \right)
      \in
      \bG\wtimes \cC E_{n-1}\bG,
    \end{equation*}
    where the codomain is equipped with its diagonal action and
    \begin{equation*}
      t
      :=
      \frac{\varepsilon}{1-\varepsilon}\cdot \frac {1-t_0}{t_0}
      \in [0,1]
      \quad\text{and}\quad
      s_i :=
      \begin{cases}
        \frac{t_i}{1-t_0}
        & \text{if }t_0<1\\
        0
        & \text{otherwise}. 
      \end{cases}
    \end{equation*}    
    That in turn transfers to the left-hand translation action by \Cref{le:diag.l} (and induction on $n$, ensuring that the earlier $E_{n-1}$ are $\bG$-spaces), whence the conclusion by the assumed coincidence $\bG\wtimes \cC E_{n-1}\bG \cong \bG\wtimes \cC E_{n-1}\bG$.

  \item\textbf{\Cref{item:th:lc.a0.cpct:eg} $\Rightarrow$ \Cref{item:th:lc.a0.cpct:en} $\Rightarrow$ \Cref{item:th:lc.a0.cpct:e1}} are obvious.
    
  \item\textbf{\Cref{item:th:lc.a0.cpct:e1} $\Leftrightarrow$ \Cref{item:th:lc.a0.cpct:gcg.diag}:} For the forward implication $(\Rightarrow)$ restrict the assumed continuous action to the closed subspace
    \begin{equation*}
      E_1\bG_{t\ge \frac 12}
      :=
      \left\{tg_1+(1-t)g_2\in \bG*\bG\ :\ t\ge \frac 12\right\}
      \subseteq
      \bG*\bG
      =
      E_1\bG,
    \end{equation*}
    and identify that space $\bG$-equivariantly with the $\bG\wtimes \cC\bG$ via
    \begin{equation*}
      E_1\bG_{t\ge \frac 12}
      \ni
      tg_1+(1-t)g_2
      \xmapsto{\quad}
      \left(g_1,\ \frac{1-t}{t}g_2\right)
      \in
      \bG\wtimes\cC\bG.
    \end{equation*}
    Conversely, continuous actions on $E_1\bG_{t\ge \frac 12}$ and the analogously-defined $E_1\bG_{t\le \frac 12}$ will glue to one on $E_1\bG$.    
    
  \item\textbf{\Cref{item:th:lc.a0.cpct:gcg.diag} $\Leftrightarrow$ \Cref{item:th:lc.a0.cpct:gcg.l}} is a direct application of \Cref{le:diag.l}.

  \item\textbf{\Cref{item:th:lc.a0.cpct:gcg.diag} $\Rightarrow$ \Cref{item:th:lc.a0.cpct:cn}:} The second projection $\bG\wtimes\cC\bG\xrightarrowdbl{\hspace{0pt}} \cC\bG$ is equivariant if the domain is equipped with its diagonal action, and the quotient topology it induces is precisely the original quotient topology on the cone.      

  \item\textbf{\Cref{item:th:lc.a0.cpct:gcg.l} $\Rightarrow$ \Cref{item:th:lc.a0.cpct:gca0.l}:} We consider two possibilities in turn.
    
    \begin{enumerate}[(i),wide]
    \item\textbf{$\bG$ is countably compact.} In that case the product and quotient topologies on $\bG\times \cC\aleph_0$ agree (implying the desired conclusion): \emph{local} countable compactness suffices, per \Cref{pr:cpctk}. 
      
    \item\textbf{$\bG$ is not countably compact.} $\bG$ will then contain (\cite[p.19]{ss_countertop}, \cite[\S 28, Exercise 4]{mnk}, \cite[Theorem 3.10.3]{eng_top_1989}, etc.) a countably-infinite discrete closed subset identifiable with $\aleph_0$, so that the action in \Cref{item:th:lc.a0.cpct:gcg.l} restricts to that of \Cref{item:th:lc.a0.cpct:gca0.l}.

    \end{enumerate}

  \item\textbf{\Cref{item:th:lc.a0.cpct:gca0.l} $\Rightarrow$ \Cref{item:th:lc.a0.cpct:lcc} ($\bG$ first-countable):} Let
    \begin{equation}\label{eq:loc.cl.bas}
      W_1\supseteq \cdots \supseteq W_n\cdots
      \ni
      1\in \bG
    \end{equation}
    be a closed-neighborhood basis, with no $W_n$ countably compact. The latter condition ensures the existence of countable open covers
    \begin{equation}\label{eq:nf.cov}
      W_n\subseteq \bigcup_{m\ge 1} U_{nm}
      ,\quad
      U_{nm}
      =
      \overset{\circ}{U}_{nm}
      \subseteq \bG
    \end{equation}
    with no finite subcovers, hence open neighborhoods
    \begin{equation*}
      U'_n
      :=
      \left(\left(\bG\setminus W_n\right)\times I\right)
      \cup
      \bigcup_{m\ge 1}
      \left(U_{nm}\times \left[0,\frac 1m\right)\right)
      \subseteq
      \bG\times I
    \end{equation*}
    of $\bG\times \{0\}\subset \bG\times I$. The image of $\bigsqcup_n U'_n$ through
    \begin{equation*}
      \bigsqcup_n \bG\times I
      \cong \bG\times \aleph_0\times I
      \xrightarrowdbl{\quad}
      \bG\times \cC\aleph_0
    \end{equation*}
    is an open neighborhood of $\bG\times \{*\}$ ($*=$ cone-tip) in $\bG\wtimes \cC\aleph_0$. Given that \Cref{eq:loc.cl.bas} is a neighborhood basis and \Cref{eq:nf.cov} have no finite subcovers, the construction ensures that for any neighborhood $V\ni 1\in \bG$, no matter how small,
    \begin{equation*}
      \exists\left(n\in \bZ_{>0}\right)
      \forall\left(\varepsilon > 0\right)
      \bigg(
      V
      \cdot
      \left(
        V\times \left[0,\varepsilon\right)
      \right)
      \ 
      \not\subseteq
      \ 
      U'_n
      \bigg)
    \end{equation*}
    (where `$\cdot$' denotes the $\bG$-action on $\bG\times I$). This means precisely that the action
    \begin{equation*}
      \bG
      \times
      \bG\wtimes \cC\aleph_0
      \xrightarrow{\quad}
      \bG\wtimes \cC\aleph_0
    \end{equation*}
    is discontinuous at $(1,*)$. 
    
  \item\textbf{\Cref{item:th:lc.a0.cpct:cn} $\Rightarrow$ \Cref{item:th:lc.a0.cpct:lcc}  ($\bG$ first-countable):} The argument again verifies the contrapositive claim, with the first-countability assumption still in place. The argument is a modified version of the preceding section of the proof: in addition to the local basis \Cref{eq:loc.cl.bas} consider also a discrete, closed, countable set $\{g_n\}_{n\in \bZ_{>0}}$ (afforded \cite[Theorem 3.10.3]{eng_top_1989} by $\bG$'s lack of global countable compactness). The $W_n\ni 1$ can be chosen sufficiently small to ensure that 
    \begin{equation}\label{eq:sqcup.wn}
      \bigsqcup_n g_n W_n = \overline{\bigsqcup_n g_n W_n}\subseteq \bG
    \end{equation}
    (i.e. the union is disjoint and closed). Indeed,
    \begin{itemize}[wide]
    \item disjointness is easily arranged for recursively, given regularity;

    \item while a cluster point $g$ of \Cref{eq:sqcup.wn} will be a cluster point for $\{g_n\}$ (contradicting the non-existence of such):
      \begin{equation*}
        \forall\left(\text{nbhds }V,V'\ni 1\right)
        \bigg(
        V\cdot W_n^{-1}\subseteq V'
        \xRightarrow{\quad g_n W_n\cap gV\ne \emptyset\quad}
        gV'
        \ni
        g_n
        \bigg).
      \end{equation*}
    \end{itemize}
    We now proceed much as before with a few minor modifications:
    \begin{itemize}[wide]
    \item in place of \Cref{eq:nf.cov} we fix countable open covers
      \begin{equation*}
        g_n W_n\subseteq \bigcup_{m\ge 1} U_{nm}
        ,\quad
        U_{nm}
        =
        \overset{\circ}{U}_{nm}
        \subseteq \bG
      \end{equation*}
      with no finite subcovers;

    \item set

      \begin{equation*}
        U'
        :=
        \left(\left(\bG\setminus \bigcup_n g_n W_n\right)\times I\right)
        \cup
        \bigcup_{n,m\ge 1}
        \left(U_{nm}\times \left[0,\frac 1m\right)\right)
        \subseteq
        \bG\times I
      \end{equation*}
      (the preimage in $\bG\times I$ of a neighborhood of the tip $*\in \cC\bG$);
      
    \item and observe that for any neighborhood $V\ni 1\in \bG$
      \begin{equation*}
        \exists\left(n\in \bZ_{>0}\right)
        \forall\left(\varepsilon > 0\right)
        \bigg(
        V
        \cdot
        \left(
          g_n V\times \left[0,\varepsilon\right)
        \right)
        \ 
        \not\subseteq
        \ 
        U'
        \bigg),
      \end{equation*}
      so that the action on the cone cannot be continuous at $(1,*)\in \bG\times \cC\bG$.
    \end{itemize}
    Finally,
  \item\textbf{\Cref{item:th:lc.a0.cpct:lcc} $\Rightarrow$ \Cref{item:th:lc.a0.cpct:eg}:} This is more comfortably outsourced to \Cref{th:lcc.eg}.  \qedhere
  \end{enumerate}
\end{th:lc.a0.cpct}

Henceforth, $\cN(\bullet)$ denotes the neighborhood filter of a point (or more generally, subset of a topological space).

\begin{theorem}\label{th:lcc.eg}
  A locally countably-compact Hausdorff group $\bG$ acts continuously on $\left(E\bG,\tau_{\varinjlim}\right)$. 
\end{theorem}
\begin{proof}
  We know from the already-settled implication \Cref{item:th:lc.a0.cpct:lcc} $\Rightarrow$ \Cref{item:th:lc.a0.cpct:en} of \Cref{th:lc.a0.cpct} that the truncated actions on the individual $E_n\bG$, $n\in \bZ_{\ge 0}$ are continuous. The ambient setup consists of
  \begin{itemize}[wide]
  \item a point
    \begin{equation*}
      x\in E_{n_0}\bG
      ,\quad
      n_0\in \bZ_{\ge 0};
    \end{equation*}
  \item an open neighborhood there of in $\left(E\bG,\tau_{\varinjlim}\right)$, consisting (essentially by definition) of a sequence of open sets
    \begin{equation}\label{eq:unum}
      \cN(x)\ni U_n\subseteq E_n\bG
      \quad\text{with}\quad
      U_n\cap E_m\bG= U_m,\ \forall m\le n;
    \end{equation}

  \item and the task of proving the existence of an origin neighborhood $\cN(1)\ni V\subseteq \bG$ and neighborhoods $\cN(x)\ni V_n\subseteq E_n\bG$ satisfying the analogue of \Cref{eq:unum} such that $V\triangleright V_n\subseteq U_n$ for all $n$. 
  \end{itemize}
  The noted continuity of the restricted actions $\triangleright_n:=\triangleright|_{\bG\times E_n\bG}$ ensures that the $V_n$ exist individually; the issue is the compatibility constraint $V_n\cap E_m\bG=V_m$. We will argue by recursion: assuming $V_n$ chosen for some sufficiently large $n$, and indicating ambient spaces housing neighborhoods by subscripts in $\cN_{\bullet}$, it will suffice to argue that
  \begin{equation*}
    \begin{aligned}
      &\forall\left(V_n\in \cN_{E_n\bG}\ :\ V\triangleright V_n\subseteq U_n\right)\\
      &\exists\left(V_{n+1}\in \cN_{E_{n+1}\bG}\right)
    \end{aligned}
    \quad
    \bigg(
    V_{n+1}\cap E_n\bG=V_n
    \ \wedge\ 
    V\triangleright V_{n+1}\subseteq U_{n+1}
    \bigg).
  \end{equation*}  
  The notation \Cref{eq:eieng} applies to $E\bG$ as well as the truncations $E_n\bG$, and there is no loss in assuming we are acting diagonally on
  \begin{equation}\label{eq:gceg}
    \begin{aligned}
      \tensor*[_{i\uparrow\varepsilon}]{E}{}\bG
      \ 
      &\cong
        \ 
        \bG\wtimes \cC E\bG
        \
        \overset{\text{\Cref{pr:cpctk}}}{\cong}\bG\times \cC E\bG\\
      &\cong
        \ 
        \varinjlim_n
        \left(\bG\wtimes \cC E_n\bG\right)
        \quad
        \cong
        \quad
        \varinjlim_n
        \left(\bG\times \cC E_n\bG\right).
    \end{aligned}    
  \end{equation}
  Moreover, since the $n\to n+1$ transition maps intertwine the corresponding maps \Cref{eq:diag2l} (one instance for each $X=\cC E_n\bG$, $n\in \bZ_{\ge 0}$), we have further switched to the left-hand translation action on \Cref{eq:gceg}. In that setting, though, the desired result follows swiftly: having selected the neighborhood $V_{n_0}\in \cN(x)$ in $\bG\times \cC E_{n_0-1}\bG$, simply extend it recursively to higher $\bG\times \cC E_{n}\bG$ arbitrarily subject only to the restriction that $V_n\subseteq U_n$. 
\end{proof}

It is perhaps worth noting that the mutual equivalence of the conditions listed in \Cref{th:lc.a0.cpct} (specifically, the implication \Cref{item:th:lc.a0.cpct:a0.same.top} $\Rightarrow$ \Cref{item:th:lc.a0.cpct:lcc}) does require \emph{some} constraint: first-countability cannot be removed entirely. To see this, we first need the following observation giving a lattice-theoretic criterion for $X\wtimes \cC\aleph_0\xrightarrow{\id} X\times \cC\aleph_0$ to be a homeomorphism. We omit the fairly routine proof. 

\begin{proposition}\label{pr:when.con.k}
  Let $X$ be a Hausdorff topological space $X$ and $\kappa$ a cardinal, regarded as a discrete topological space. The product and quotient topologies on $X\times \cC\kappa$ agree precisely when, for every $x\in X$ and countable collection $\cU=\left(U_{\sigma n}\right)_{\sigma<\kappa, n\in \bZ_{>0}}$ of open sets in $X$, the condition
  \begin{equation}\label{eq:when.con.k}
    \begin{aligned}
      &\exists\left(U\in\cN(x)\right)
        \forall\left(\sigma<\kappa\right)
        \left(
        U
        \subseteq 
        \bigcup_{n}U_{\sigma n}
        \right)
        \xRightarrow{\quad}\\
      \xRightarrow{\quad}\ 
      &\exists\left(V\in \cN(x)\right)
        \forall\left(\sigma<\kappa\right)
        \exists\left(M_{\sigma}\in \bZ_{>0}\right)
        \forall\left(\sigma<\kappa\right)
        \left(
        V
        \subseteq 
        \bigcup_{n=1}^{M_{\sigma}}U_{\sigma n}
        \right)
    \end{aligned}    
  \end{equation}
  holds.  \qedhere
\end{proposition}

In particular, we have the following consequence. 

\begin{corollary}\label{cor:when.con.k}
  For Hausdorff $X$ and a cardinal $\kappa$ the product and quotient topologies on $X\times \cC\kappa$ agree whenever either of the following conditions holds:
  \begin{enumerate}[(a),wide]
  \item\label{item:cor:when.con.k:l.a0.cpct} $X$ is locally countably-compact.

  \item\label{item:cor:when.con.k:k.int.nbhd} For every $x\in X$ the set $\cN(x)$ of $x$-neighborhoods is closed under $\kappa$-fold intersections. 
  \end{enumerate}
\end{corollary}
\begin{proof}
  In case \Cref{item:cor:when.con.k:l.a0.cpct} a locally-compact neighborhood $V\in \cN(x)$ will verify \Cref{eq:when.con.k}. In case \Cref{item:cor:when.con.k:k.int.nbhd}, on the other hand, $x$ belongs to some $U_{\sigma n_{\sigma,x}}$ for arbitrary $\sigma$; simply set $V:=\bigcap_{\sigma}U_{\sigma,n_{\sigma,x}}$ to conclude.
\end{proof}

Via \Cref{cor:when.con.k}'s \Cref{item:cor:when.con.k:k.int.nbhd} branch, \Cref{ex:cnt.int.nbhd} shows that for $\kappa=\aleph_0$ \Cref{cor:when.con.k}\Cref{item:cor:when.con.k:k.int.nbhd} (incompatible with first-countability save for discrete spaces) can certainly hold for topological groups.

\begin{example}\label{ex:cnt.int.nbhd}
  The goal is to exhibit non-locally-countably-compact topological groups $\bG$ with the property that every countable intersection of identity neighborhoods is another such. $\bG$ will be totally-ordered abelian groups equipped with the \emph{order (or open-interval) topology} (automatically a group topology \cite[\S II.8, post Theorem 11]{fuchs_ord_1963}).
  
  Condition \Cref{cor:when.con.k}\Cref{item:cor:when.con.k:k.int.nbhd} above, in the context of ordered abelian $(\bG,+,<)$, translates to countable sets of strictly positive elements having strictly positive lower bounds. It suffices, at that point, to take for $\bG$ any \emph{$\eta_1$-group} in the terminology of \cite[Definition 1.37(iii)]{dw_super-real_1996}: for subsets $S_i\subseteq \bG$ of at-most-countable total cardinality we have
  \begin{equation}\label{eq:s.gap}
    \forall\left(s_i\in S_i,\ i=1,2\right)\left(s_1<s_2\right)
    \xRightarrow{\quad}
    \exists(g\in \bG)
    \forall \left(s_i\in S_i\right)
    \left(s_1<g<s_2\right).
  \end{equation}
  Per \cite[Theorem 4.29]{dw_super-real_1996}, concrete examples are provided by the \emph{ultrapowers} \cite[Definition 4.18]{dw_super-real_1996} $\bR^{\kappa}/\cU$ for cardinals $\kappa$ and \emph{$\aleph_1$-incomplete ultrafilters} \cite[Definition 6.6.3]{gld_ultra_2022} $\cU$ on $\kappa$: the incompleteness condition, meaning that $\cU$ is not closed under countable intersections, is precisely equivalent \cite[Proposition 5(ii)]{MR1340464} to $\bR\lhook\to \bR^{\kappa}/\cU$ being proper.

  Failure of local countable compactness is also immediate: the standard (diagonal) copy $\bR\subset \bR^{\kappa}/\cU$ is discrete in the inherited order topology, hence the discreteness of the infinite closed subset
  \begin{equation*}
    \left\{t\varepsilon\ :\ -1<t<1\in \bR\in \right\}
    \subseteq
    [-\varepsilon,\varepsilon]
  \end{equation*}
  for arbitrarily small $0<\varepsilon\in \bR^{\kappa}/\cU$. 
\end{example}

\begin{remarks}\label{res:int.nbhd.enough}
  \begin{enumerate}[(1),wide]
  \item\label{item:res:int.nbhd.enough:other.flds} The crux of the matter, in \Cref{ex:cnt.int.nbhd}, is that \Cref{eq:s.gap} for arbitrary $\left|S_1\cup S_2\right|<\aleph_1$ entails \Cref{eq:when.con.k}. For that reason, many variations on the example are possible: $\bQ^{\kappa}/\cU$ will do just as well, for instance (for an $\aleph_1$-incomplete ultrafilter $\cU$ on $\kappa$), or indeed $\bK^{\kappa}/\cU$ for \emph{any} subfield $\bK\le \bR$.

  \item\label{item:res:int.nbhd.enough:not.pcpct} Not only are the groups $\left(\bK^{\kappa}/\cU,+\right)$ in \Cref{item:res:int.nbhd.enough:other.flds} above not locally countably-compact, but they are in fact not even locally pseudocompact: for closed neighborhoods $[-\varepsilon,\varepsilon]\in \cN(0)$
    \begin{itemize}[wide]
    \item define an arbitrary unbounded real-valued function on the closed discrete subset $\{t\varepsilon\}_{|t|\le 1}$;
    \item and extend that function continuously to all of $[-\varepsilon,\varepsilon]$ by \emph{Tietze} \cite[Theorem 15.8]{wil_top}, using the fact \cite[\S II.8, post Proposition 12]{fuchs_ord_1963} that the interval topology on a totally-ordered abelian group is normal.
    \end{itemize}
  \end{enumerate}  
\end{remarks}

Property \Cref{item:th:lc.a0.cpct:cn} in \Cref{th:lc.a0.cpct} is also easily characterized in terms of countable covers, in an analogue of \Cref{pr:when.con.k}.

\begin{lemma}\label{le:when.act.on.cn.cont}
  A Hausdorff topological group $\bG$ acts continuously on its quotient-topologized cone $\cC \bG$ precisely when,
  \begin{equation}\label{eq:when.act.on.cn.cont}
    \forall\left(\bigcup_n \left(U_n=\overset{\circ}{U}_n\right)=\bG\right)
    \exists\left(V\in \cN(1)\right)
    \forall\left(g\in \bG\right)
    \left(gV\subseteq \bigcup^{\text{finite union}}U_n\right).
  \end{equation}
  \qedhere
\end{lemma}

\Cref{eq:when.act.on.cn.cont} can be regarded as a kind of uniform local countable compactness, with the uniformity tailored to the cover. \Cref{le:when.act.on.cn.cont} shows that \Cref{ex:cnt.int.nbhd} does somewhat more than initially claimed: not only does \Cref{th:lc.a0.cpct}'s \Cref{item:th:lc.a0.cpct:a0.same.top} not (absent first-countability) imply \Cref{item:th:lc.a0.cpct:lcc}, but it does not even imply the weaker \Cref{item:th:lc.a0.cpct:cn}.

\begin{corollary}\label{cor:int.nbhd.not.act.con}
  For an $\aleph_1$-incomplete ultrafilter $\cU\subseteq 2^{\aleph_0}$ the order-topologized additive group $\bG:=\bR^{\aleph_0}/\cU$ does not act continuously on $\cC\bG$. 
\end{corollary}
\begin{proof} 
  The countable open cover $\bG=\bigcup_n U_n$ meant to negate \Cref{eq:when.act.on.cn.cont} will be of the form $U_n:=f^{-1}\left(\bR_{>\frac 1n}\right)$ for a continuous function $\bG\xrightarrow{f}\bR_{>0}$ we spend the rest of the proof constructing; or rather, it will be convenient to construct
  \begin{equation*}
    \bG
    \xrightarrow[\quad\text{continuous unbounded}\quad]{\quad 1/f\quad}
    \bR_{\ge 1}
  \end{equation*}
  instead.
  
  The character $\chi(\bG)$ (as in \Cref{eq:chars}) is easily seen to be precisely $\fc:=2^{\aleph_0}$ (and cannot, at any rate, be larger, given that $|\bG|=\fc$ to begin with). There is thus a local closed-neighborhood base
  \begin{equation*}
    (W_{\sigma})_{\sigma<\fc}
    ,\quad
    W_{\sigma}=\left[-\varepsilon_{\sigma},\varepsilon_{\sigma}\right]
    ,\quad
    \begin{aligned}
      \varepsilon_{\sigma}\in \bG_{>0}
      \text{ \emph{infinitesimal} }
      &\text{\cite[Definition 2.3(i)]{dw_super-real_1996}}:\\
      \forall\left(n\in \bZ_{\ge 0}\right)
      &\left(n\varepsilon_{\sigma}<1\right).
    \end{aligned}    
  \end{equation*}
  For a $\fc$-enumeration $\left\{g_{\sigma}\right\}_{\sigma<\fc}$ of $\bR\subset \bG$ define $1/f$ arbitrarily on the closed subset
  \begin{equation*}
    \bigcup_{\sigma<\fc}
    \left(W'_{\sigma}:=g_{\sigma}+W_{\sigma}\right)
    \subset
    \bG
  \end{equation*}
  so as to ensure that
  \begin{itemize}[wide]
  \item all restrictions $\left(1/f\right)|_{W'_{\sigma}}$ are $\bR_{\ge 1}$-valued, continuous and unbounded (always possible: \Cref{res:int.nbhd.enough}\Cref{item:res:int.nbhd.enough:not.pcpct});

  \item and evaluate to $1$ at the endpoints $g_{\sigma}\pm \varepsilon_{\sigma}$ of the closed intervals $W'_{\sigma}$. 
  \end{itemize}
  The latter condition then permits the continuous extension of $1/f$ thus defined to all of $\bG$ by simply setting $(1/f)|_{\bG\setminus \bigcup W'_{\sigma}}\equiv 1$.

  That the open cover by $U_n:=f^{-1}\left(\left(\frac 1n,1\right]\right)$ for $f$ thus built fails to satisfy \Cref{eq:when.act.on.cn.cont} is immediate from the very construction: no matter how small the candidate neighborhood $V:=W_{\sigma}\in \cN(1)$ is, $g_{\sigma'}+W_{\sigma}$ is not covered by finitely many $U_n$ for smaller $W_{\sigma'}\subset W_{\sigma}$ because ($1/f$ being unbounded on the smaller set $W'_{\sigma'}=g_{\sigma'}+W_{\sigma'}$) $f|_{g_{\sigma'}+W_{\sigma}}$ is not bounded away from $0$. 
\end{proof}

\begin{remarks}\label{res:is.cnt.parcpct}
  \begin{enumerate}[(1),wide]
  \item In a way, the choice of $U_n:=f^{-1}\left(\bR_{>1/n}\right)$ for an open cover in the proof of \Cref{cor:int.nbhd.not.act.con} was inevitable: all countable open covers are effectively of that form, in that any
    \begin{equation*}
      \bigcup_n \left(U_n=\overset{\circ}{U}_n\right)=\bG:=\bR^{\kappa}/\cU
    \end{equation*}
    has an open \emph{refinement} \cite[\S 3.1]{eng_top_1989}
    \begin{equation*}
      \bigcup_n \left(V_n=\overset{\circ}{V}_n\right)=\bG
      \quad \text{for}\quad
      \left[
        \begin{aligned}
          &f^{-1}\left(\bR_{>1/n}\right) = V_n\subseteq U_n\\
          &\bG\xrightarrow[\quad\text{continuous}\quad]{\quad f\quad}\bR_{>0}
        \end{aligned}
      \right.
    \end{equation*}
    Indeed:
    \begin{itemize}[wide]
    \item being a totally-ordered space equipped with its order topology, $\bG$ is both (even \emph{hereditarily}) normal \cite[Corollary 3.2]{MR564104} and \emph{countably paracompact} \cite[Theorem 3.6]{MR564104} (in fact, $\bG:=\bR^{\kappa}/\cU$ is even paracompact: \cite[Theorem 6.1]{MR458351});
      
    \item normality+countable paracompactness is in turn equivalent \cite[Theorem 4]{zbMATH03065554} to the existence of a sandwiched continuous function $f_{\uparrow} < f <f_{\downarrow}$ for any pair
      \begin{equation*}
        \bG
        \xrightarrow[\quad f_{\uparrow},\ f_{\downarrow}\ \text{upper/lower semicontinuous respectively}\quad]{\quad f_{\updownarrow} \quad}
        \bR;
      \end{equation*}

    \item so the claim follows by sandwiching the desired continuous function $f$ between $f_{\uparrow}\equiv 0$ and the function $f_{\downarrow}$ defined implicitly by
      \begin{equation*}
        \mathrm{graph}\left(f_{\downarrow}\right)
        :=
        \bigcup_{n\ge 1}\left(U_n\setminus \bigcup_{1\le m<n}U_m\right)\times \left\{\frac 1n\right\}
      \end{equation*}
      (the lower semicontinuity of $f_{\downarrow}$ is immediate from its definition).
    \end{itemize}

  \item In reference to the (full) paracompactness of $\bG:=\bR^{\kappa}/\cU$ noted in passing in the preceding item, it is apposite to point out that in fact all totally ordered groups equipped with their order topology are so: the \emph{left uniformity} \cite[Problem 35F]{wil_top} has what \cite{MR315664} refers to as a \emph{totally-ordered base}
    \begin{equation*}
      \left\{(x,y)\in \bG^2\ :\ x^{-1}y\in \left(g^{-1},g\right)\right\}
      \subseteq
      \bG^2
      ,\quad
      g\in \bG_{>1},
    \end{equation*}
    hence the conclusion by that paper's main result. 
  \end{enumerate}  
\end{remarks}


\addcontentsline{toc}{section}{References}

\def\polhk#1{\setbox0=\hbox{#1}{\ooalign{\hidewidth
  \lower1.5ex\hbox{`}\hidewidth\crcr\unhbox0}}}
  \def\polhk#1{\setbox0=\hbox{#1}{\ooalign{\hidewidth
  \lower1.5ex\hbox{`}\hidewidth\crcr\unhbox0}}}
  \def\polhk#1{\setbox0=\hbox{#1}{\ooalign{\hidewidth
  \lower1.5ex\hbox{`}\hidewidth\crcr\unhbox0}}}
  \def\polhk#1{\setbox0=\hbox{#1}{\ooalign{\hidewidth
  \lower1.5ex\hbox{`}\hidewidth\crcr\unhbox0}}}
  \def\polhk#1{\setbox0=\hbox{#1}{\ooalign{\hidewidth
  \lower1.5ex\hbox{`}\hidewidth\crcr\unhbox0}}}
  \def\polhk#1{\setbox0=\hbox{#1}{\ooalign{\hidewidth
  \lower1.5ex\hbox{`}\hidewidth\crcr\unhbox0}}}


\Addresses

\end{document}